\documentclass[11pt]{amsproc}

\usepackage{amsmath}
\usepackage{amsthm}
\usepackage{amssymb}
\usepackage{graphicx}
\usepackage{stackrel}
\usepackage{fullpage}
\usepackage{pdfpages}

\usepackage[]{hyperref}
\hypersetup{colorlinks=false}

\def\R{{{\mathbb R}}}

\def\M{{{\mathcal M}}}
\def\normal{{{\vec{n}}}}
\def\d{{{\partial}}}
\newcommand{\inprod}[2]{\left\langle #1 \, , \, #2 \right\rangle}

\newtheorem{theorem}{Theorem}[section]
\newtheorem{lemma}[theorem]{Lemma}
\newtheorem{proposition}[theorem]{Proposition}

\numberwithin{equation}{section}

\title{A Pinching Estimate for Convex Hypersurfaces Evolving Under a Nonhomogeneous Variant of Mean Curvature Flow}
\author{Tim Espin\footnote{Tim Espin was supported by The Maxwell Institute Graduate School in Analysis and its Applications, a Centre for Doctoral Training funded by the UK Engineering and Physical Sciences Research Council (grant EP/L016508/01), the Scottish Funding Council, Heriot-Watt University and the University of Edinburgh. }}
\address{Tim Espin, Maxwell Institute for Mathematical Sciences, School of Mathematics, University of Edinburgh, Edinburgh, UK, EH9 3FD}
\email{Tim.Espin@ed.ac.uk}

\begin{document}

\begin{abstract}
	We study a variant of the mean curvature flow for closed, convex hypersurfaces where the normal velocity is a nonhomogeneous function of the principal curvatures. We show that if the initial hypersurface satisfies a certain pinching condition, then this is preserved and the flow converges to a sphere under rescaling.
\end{abstract}

\maketitle
	
\section{Introduction}

Let $\M^n$ be a closed, orientable $n$-dimensional smooth manifold, and let $F_0 : \M^n \longrightarrow \R^{n+1}$ be a strictly convex immersion of $\M^n$ as a hypersurface into Euclidean space. We study the curvature flow given by
\begin{equation}\label{TheEquation}
\begin{cases}
\d_t F(p,t) = -f(H)\normal(p,t) \, , & \ \ \ \ t \geq 0 \\
F(p,0) = F_0(p) \, , &
\end{cases}
\end{equation} where $\normal(p,t)$ is the unit outer normal vector, $H$ is the mean curvature, and $f:\R^+ \longrightarrow \R^+$ is a smooth, convex function satisfying the conditions
\begin{subequations}
	\begin{align}
	& f(0) \geq 0 \, , \label{Positive} \\
	& f'(x) > 0 \ \ \forall x \geq 0 \, . \label{Increasing}
	\end{align}
\end{subequations}
The surface $F(\cdot,t)$ is denoted by $\M_t$. \\

We are interested specifically in the function \[
f(H) = H(\ln \hat{H})^\alpha \, ,
\] where $\alpha > 0$ and $\hat{H} = H + H_0$, with $H_0 \geq e$ a constant chosen so that $f$ satisfies our conditions \eqref{Positive} and \eqref{Increasing}. Note that this is a nonhomogeneous function of the principal curvatures, unlike many previously studied curvature flows (see for example \cite{Andrews},\cite{Chow},\cite{Huisken1984}). This choice of $f$ satisfies
\begin{subequations}
	\begin{align}
	f' & = (\ln\hat{H})^{\alpha-1}\left[\ln\hat{H} + \alpha\frac{H}{\hat{H}}\right] \, , \\
	f'' & = \frac{\alpha}{\hat{H}}(\ln\hat{H})^{\alpha-2} \left[ \ln\hat{H} + \frac{H_0 \ln\hat{H}}{\hat{H}} + (\alpha - 1)\frac{H}{\hat{H}} \right] \, , \\
	\frac{f''}{f'} & = \frac{\alpha}{\hat{H}\ln\hat{H}}\left( 1 + \frac{H_0\ln\hat{H} - H}{\hat{H}\ln\hat{H} + \alpha H} \right) \, , \\
	Hf' - f & = \alpha\frac{H^2}{\hat{H}}(\ln\hat{H})^{\alpha-1} \, . \label{Hf'-f}
	\end{align}
\end{subequations}

The important properties are that \[
f', \ f'', \ Hf' - f > 0 \ \text{ and } \ \frac{Hf''}{f'} \leq 2\alpha \, ,
\] due to our choice of $H_0$.

We wish to prove that the flow given above preserves the convexity of the initial hypersurface. To do this we follow the method of Schulze in \cite{Schulze2} and consider the quantity $K/H^n$, where $K$ is the Gauss curvature. Schulze's approach was based on that of Chow in \cite{Chow}, who was the first to consider this quantity. In particular we prove the following:

\begin{theorem}
	Let $F(\cdot,t)$ be a family of hypersurfaces evolving under the curvature flow \eqref{TheEquation}. If there exists a positive constant $C(n,\alpha) < 1/n^n$ such that the initial hypersurface is pinched in the sense that \[
	\frac{K(p)}{H^n(p)} \geq C(n,\alpha) \quad \forall p \in \M_0 \, ,
	\] then this is preserved under the flow. Moreover, the family of pinched hypersurfaces converges to a sphere after rescaling.
\end{theorem}

Heuristically, the pinching condition indicates how close the ratio between the smallest and largest eigenvalues is to 1. If the constant $C$ is very close to $1/n^n$, the principal curvatures of the initial hypersurface are close to one another, and the surface is close to being a sphere. The preservation of the pinching condition therefore tells us that the hypersurfaces become more spherical as the flow progresses.

The theorem above is similar to that proved by Schulze in \cite{Schulze2} for flows of the form $f(H) = H^k, \ k \geq 1$. There are a wide range of results published on flows of convex hypersurfaces for which $f$ is a homogeneous function of the principal curvatures. Huisken's 1984 paper \cite{Huisken1984} showed that convex hypersurfaces evolving under the mean curvature flow $f(H) = H$ contract to a spherical point in finite time, and converge smoothly to a sphere after appropriate rescaling. Chow obtained similar results for a variety of other flows where the normal velocity is not necessarily a function of the mean curvature, for instance the $n$th root of the Gauss curvature in \cite{Chow}. These were extended by Andrews in \cite{Andrews} to a more general result on homogeneous of degree one velocities. In \cite{Chow}, Chow also proved a result for normal velocities whose degree of homogeneity is greater than 1. It was this work that led to the aforementioned paper of Schulze. The flow of convex hypersurfaces under the nonhomogeneous function $f(H) = \hat{H}/\ln\hat{H}$ was studied by Alessandroni and Sinestrari in \cite{A-S} using maximum principle methods similar to those of Huisken in \cite{Huisken1984}. However, we were not able to follow the proof of Theorem 3 in their paper and we don't believe that their techniques yield our result.

The organisation of this paper is as follows. In Section 2 we justify the short-time existence of solutions. This is followed by a statement of the well-known evolution equations for some geometric quantities under the flow \eqref{TheEquation} along with the elementary result that the flow preserves strict and mean-convexity in Section 3. In Section 4 we derive the evolution equation for $K/H^n$ and prove that if the initial hypersurface is pinched strongly enough then this pinching is preserved by the flow. Finally, in section 5 we show that under rescaling the pinched flow converges to a sphere.

\section{Short-Time Existence}

At this point we briefly justify the short-time existence of unique smooth solutions to equation \eqref{TheEquation}, using results on nonlinear parabolic PDEs from \cite{GerhardtBook}. We begin by writing $\M^n$ locally in coordinate patches as a graph over $\Omega \subset \R^n$. More precisely, suppose that locally $\M^n = \text{graph}(u_0)$, $u_0 : \Omega \longrightarrow \R$, and $F$ is given by $F(p,t) = (p,u(p,t))$. In this setting \eqref{TheEquation} reduces to the initial value problem
\begin{equation}\label{TheEquationGraph1}
\begin{cases}
\d_t u = -\sqrt{1 + |Du|^2} \, f\left( -D\cdot\frac{Du}{\sqrt{1+ |Du|^2}} \right) \\
u(p,0) = u_0(p) \, ,
\end{cases}
\end{equation}
where $D$ and $D\cdot$ are the usual gradient and divergence on $\R^n$ and $u_0$ is smooth. This comes from the expressions given in Appendix A of \cite{EckerBook} for the mean curvature of a graph $u$, \[
H = -D\cdot\frac{Du}{\sqrt{1+ |Du|^2}} = -\frac{1}{\sqrt{1 + |Du|^2}}\left( \Delta u - \frac{Du\cdot (D^2 u) Du}{1+ |Du|^2} \right) \, ,
\] where $D^2u$ is the Hessian matrix of $u$, and the outer (upper pointing) normal \[
\normal = \frac{(-Du , 1)}{\sqrt{1 + |Du|^2}} \, .
\] Moreover, the metric and inverse metric on the graph $u$ are given by \[
g_{ij} = \delta_{ij} + \d_i u \d_j u \, , \qquad g^{ij} = \delta^{ij} - \frac{\d_i u \d_j u}{1 + |Du|^2} \, ,
\] so we can rewrite \eqref{TheEquationGraph1} as
\begin{equation}\label{TheEquationGraph2}
\d_t u = -\sqrt{1 + |Du|^2} \, f\left( \frac{-g^{ij}D^2_{ij}u}{\sqrt{1 + |Du|^2}} \right) \, .
\end{equation}

In particular, note that \[
\frac{d}{d(g^{ij}D^2_{ij}u)} \sqrt{1 + |Du|^2} \, f\left( \frac{-g^{ij}D^2_{ij}u}{\sqrt{1 + |Du|^2}} \right) = -f'\left( \frac{-g^{ij}D^2_{ij}u}{\sqrt{1 + |Du|^2}} \right) < 0
\] by virtue of \eqref{Increasing}. We can therefore invoke theorems 2.5.7 and 2.5.9 from \cite{GerhardtBook} which guarantee the existence of a unique smooth solution to \eqref{TheEquationGraph1} on a short time interval. This can then be extended in the usual way via a compactness argument to a full solution of \eqref{TheEquation}.

\section{Evolution Equations for Geometric Quantities}

The following proposition is well known for flows of the type \eqref{TheEquation}. See for example Appendix A of \cite{EckerBook} or \cite{ZhuBook} for outlines of the proofs in the mean curvature flow case.

\begin{proposition}\label{TimeChange}
	For surfaces evolving under \eqref{TheEquation} we have the following evolution equations for geometric quantities:
	\begin{subequations}
		\begin{align}
		\frac{\d}{\d t} g_{ij} & = -2f(H)h_{ij} \, , \\
		\frac{\d}{\d t} g^{ij} & = 2f(H)g^{ik}g^{jl}h_{kl} \, , \\
		\label{TimeChange3} \frac{\d}{\d t} h_{ij} & = f'(H)\Delta_{\M_t}h_{ij} + f''(H)\nabla_i H \nabla_j H - \big(f(H) + Hf'(H)\big)h_{ik}g^{kl}h_{lj} + f'(H)|A|^2h_{ij} \, , \\
		\label{TimeChange4} \frac{\d}{\d t} H \ & = f'(H)\Delta_{\M_t} H + f''(H) |\nabla H|^2 + f(H)|A|^2 \, , \\
		\frac{\d}{\d t} |A|^2 & = f'(H)\big[ \Delta_{\M_t}|A|^2 - 2|\nabla A|^2 + 2|A|^4 \big] + 2f''(H)g^{ij}g^{kl}h_{ik}(\nabla_j H)(\nabla_l H) \\
		& \nonumber \qquad + 2\text{tr}(A^3) \big[ f(H) - Hf'(H) \big] \, , \\
		\frac{\d}{\d t} \sqrt{g} & = -\sqrt{g}Hf(H)	\, .
		\end{align}
	\end{subequations}
\end{proposition}

By applying Hamilton's Maximum Principle for tensors (Theorem 4.1 in \cite{Huisken1984}) to the evolution equation for $h_{ij}$ \eqref{TimeChange3}, and the usual parabolic maximum principle to the evolution equation for $H$ \eqref{TimeChange4}, we find that both convexity and strict mean-convexity are preserved by the flow:
\begin{proposition}
	Suppose $F_0$ is a hypersurface with everywhere-positive mean curvature $H(F_0) \geq \delta > 0$. Then $H(F(\cdot,t)) \geq \delta$ for all $t$. Moreover, if $F_0$ is convex then $h_{ij} \geq 0$ is preserved by the flow.
\end{proposition}

\begin{proof}
	The result $H(F(\cdot,t)) \geq \delta$ for all $t$ follows from applying the parabolic maximum principle to \eqref{TimeChange4} in Proposition \ref{TimeChange}. To prove that convexity is preserved, we use the equation \eqref{TimeChange3}. In order to use Proposition 4.1 of \cite{Huisken1984}, set \[
	M_{ij} = h_{ij} \, , \ u^k = 0 \ \text{ and } \ N_{ij} = f''(H)\nabla_i H \nabla_j H - \big(f(H) + Hf'(H)\big)h_{ik}g^{kl}h_{lj} + f'(H)|A|^2h_{ij} \, .
	\] This $N_{ij}$ satisfies the null eigenvector condition since if $h_{ij}X^j = 0$ then \[
	N_{ij}X^iX^j = f''(H)\nabla_i H \nabla_j HX^iX^j = f''(H)(X^i\nabla_i H)^2 \geq 0 \, .
	\] The result follows from this.
\end{proof}

The importance of the above result is twofold. Firstly, it implies $H_{\min}(t)$ is increasing and the uniform parabolicity of \eqref{TheEquation} is not lost. Secondly by writing $|A|^2$ and $H^2$ in terms of the principal curvatures and expanding the square, we see that $|A|^2 \leq H^2$ for any convex surface. The preservation of convexity means that $|A|^2 \leq H^2$ on $F(\cdot,t)$ for all $t$, which will be used later. \\

The evolution equation for $H$ also leads to a finite upper bound on the maximal time of existence.
\begin{proposition}
	For surfaces evolving according to equation \eqref{TheEquation}, we have
	\begin{equation}\label{LowerBoundOnH}
	H_{\min}(t) \geq \left( \frac{1}{H_{\min}^2(0)} - \frac{2t}{n} \right)^{-1/2} \, .
	\end{equation}
	This in turn gives an upper bound on the maximal time of existence: \[
	T \leq \frac{n}{2H_{\min}^2(0)} < \infty \, .
	\]
\end{proposition}

\begin{proof}
	Since $\ln\hat{H} \geq 1$ and $|A|^2 \geq \frac{1}{n}H^2$, we have $f(H)|A|^2 \geq \frac{1}{n}H^3$. We can therefore compare \eqref{TimeChange4} in Proposition \ref{TimeChange} with the ODE \[
	\frac{d}{dt}\psi = \frac{1}{n}\psi^3 \, , \quad \psi(0) = H_{\min}(0) \, .
	\] This has solution \[
	\psi(t) = \left( \frac{1}{H_{\min}^2(0)} - \frac{2t}{n} \right)^{-1/2} \, ,
	\] and so the lower bound on $H_{\min}$ \eqref{LowerBoundOnH} follows from the comparison principle. The upper bound on the maximal existence time $T$ follows directly from \eqref{LowerBoundOnH}.
\end{proof}

In exactly the same manner as \cite{Schulze1} we obtain the result that $\max_{\M_t}|A|^2$ and $\max_{\M_t}H^2$ must both tend to infinity as the maximal time of existence $T$ is approached. If this were not so we obtain bounds on the derivatives of $|A|^2$ and $H$, giving a smooth limiting surface, contradicting the maximality of $T$.

\section{Pinching with Gauss Curvature}

Our aim during this section is to show that even though $H$ blows up near a singularity of the flow, the ratios between the principal curvatures approach 1. Following the method of Schulze in \cite{Schulze2}, we first derive an evolution equation for the quantity \[
\gamma := \frac{K}{H^n} \, .
\] By the arithmetic-geometric mean inequality we have \[
0 \leq \gamma \leq \frac{1}{n^n} \, ,
\] with equality on the right hand side if and only if all the principal curvatures are equal. \\

Throughout this section, if $\zeta$ and $\xi$ are scalars we use the notation $\inprod{\nabla\zeta}{\nabla\xi} := g^{ab}\nabla_a\zeta\nabla_b\xi$, and $|\nabla\xi|^2 = \inprod{\nabla\xi}{\nabla\xi}$. The first step is to find the evolution equation for $K$:

\begin{lemma}
	Let $K = \det[g^{ik}h_{kj}]$ be the Gauss curvature of the strictly convex surface $\M_t$. Then we have:
	\begin{subequations}
	\begin{align}
	\label{TimeK} \d_t K = {} & K \left[ H(f-Hf') + f'b^{jk}\Delta_{\M_t} h_{jk} + f''b^{jk}\nabla_j H \nabla_k H + n|A|^2f' \right] \, ,\\
	\label{NablaK} \nabla_a K = {} & Kb^{jk}\nabla_a h_{jk} \, ,\\
	\label{LapK} \Delta_{\M_t} K = {} & Kb^{jk}\Delta_{\M_t}h_{jk} + \frac{n-1}{nK}|\nabla K|^2 - \frac{K}{H^2}Y^2 + \frac{H^{2n}}{nK}\left| \nabla\gamma \right|^2 \, , \\
	\label{EvK} \d_t K = {} & f'\Delta_{\M_t} K - f'\frac{n-1}{n}\frac{|\nabla K|^2}{K} - f'\frac{H^{2n}}{nK}\left| \nabla\gamma \right|^2 + f'\frac{K}{H^2}Y^2 + f''Kb^{jk}\nabla_jH\nabla_kH \\
	& + HK(f-Hf') + f'nK|A|^2 \, , \nonumber
	\end{align}
	\end{subequations}
	where $b^i_j$ denotes the components of the inverse of the Weingarten map so that $h^i_kb^k_j = \delta^i_j$ and $Y^2 = g^{ij}b^{lp}b^{mq}(H\nabla_i h_{lm} - h_{lm}\nabla_i H)(H\nabla_j h_{pq} - h_{pq}\nabla_j H)$.
\end{lemma}

The previous lemma agrees with Lemma 2.1 in \cite{Schulze2}.

\begin{proof}
	Given that $K = \text{det}[h^i_j] = \det[g^{ik}h_{kj}]$, we make use of the following identity: If $A(t)$ is a positive definite matrix then \[
	\d_t\det[A(t)] = \det[A(t)]\text{tr}[A^{-1}(t)\d_tA(t)] \, .
	\]
	The first two equations follow from this and substituting in the evolution equation for $h_{ij}$ \eqref{TimeChange3}.
	
	For the third, taking the covariant derivative of \eqref{NablaK} gives \[
	\Delta_{\M_t}K = g^{ac}\nabla_a \nabla_c K = g^{ac}(\nabla_aK)b^{jk}\nabla_ch_{jk} + Kg^{ac}(\nabla_ab^{jk})(\nabla_ch_{jk}) + Kb^{jk}\Delta_{\M_t}h_{jk} \, ,
	\] and now we use $b^{jk}\nabla_ch_{jk} = \frac{1}{K}\nabla_c K$ and the identity \[
	Kg^{ac}(\nabla_ab^{jk})(\nabla_ch_{jk}) = -\frac{1}{nK}|\nabla K|^2 - \frac{K}{H^2}Y^2 + \frac{H^{2n}}{nK}|\nabla\gamma|^2 \, .
	\] Finally, substitute $Kb^{jk}\Delta_{\M_t}h_{jk}$ from \eqref{LapK} into \eqref{TimeK} to obtain \eqref{EvK}.
\end{proof}

\begin{lemma}\label{Evol-u-tion}
	We have the following equations for $\gamma$:
	\begin{subequations}
	\begin{align}
		\label{Evol-u-tion1} \nabla_a\gamma = {} & \frac{\nabla_a K}{H^n} - \frac{nK\nabla_a H}{H^{n+1}} \, , \\
		\label{Evol-u-tion2} \Delta_{\M_t}\gamma = {} & \frac{\Delta_{\M_t}K}{H^n} - \frac{2n}{H^{n+1}}\inprod{\nabla K}{\nabla H} - \frac{n\gamma}{H}\Delta_{\M_t}H + \frac{n(n+1)\gamma}{H^2}|\nabla H|^2 \, , \\
		\d_t\gamma = {} & f'\Delta_{\M_t}\gamma + f'\frac{n+1}{nH^n}\inprod{\nabla\gamma}{\nabla H^n} - f'\frac{n-1}{nK}\inprod{\nabla\gamma}{\nabla K} - f'\frac{H^n}{nK}|\nabla\gamma|^2 \label{Evol-u-tion3} \\
		& {} + \frac{Hf' - f}{H}(n|A|^2 - H^2)\gamma + f'\frac{Y^2}{H^2}\gamma + f''\left[ \left( b^{ij} - \frac{n}{H}g^{ij} \right)\nabla_i H\nabla_j H \right]\gamma \, .\nonumber
	\end{align}
	\end{subequations}
\end{lemma}

\begin{proof}
	The first two equations follow by directly differentiating $\gamma$. Using \eqref{Evol-u-tion1} we derive the identities
	\begin{subequations}
		\begin{align}
		\frac{nK}{H^{n+1}}|\nabla H|^2 {} & = \frac{\inprod{\nabla K}{\nabla H}}{H^n} - \frac{\inprod{\nabla \gamma}{\nabla H^n}}{nH^{n-1}} \\
		\inprod{\nabla K}{\nabla H} {} & = \frac{H^{n+1}}{nK}\left( \frac{|\nabla K|^2}{H^n} - \inprod{\nabla\gamma}{\nabla K} \right) \, .
		\end{align}
	\end{subequations}
	Applying these and $\Delta_{\M_t}H^n = nH^{n-1}\Delta_{\M_t}H + n(n-1)H^{n-2}|\nabla H|^2$ to \eqref{Evol-u-tion2} gives \[
	\Delta_{\M_t}\gamma = \frac{\Delta_{\M_t} K}{H^n} - \frac{\gamma\Delta_{\M_t} H^n}{H^n} - \frac{n+1}{nH^n} \inprod{\nabla\gamma}{\nabla H^n} +\frac{n-1}{nK}\inprod{\nabla\gamma}{\nabla K} + \frac{n(n-1)\gamma}{H^2}|\nabla H|^2 - \frac{n-1}{nK}\frac{|\nabla K|^2}{H^n} \, .
	\] We also have \[
	\d_t H^n = f'\Delta_{\M_t}H^n + nH^{n-1}|A|^2f + nH^{n-2}\left( Hf'' - (n-1)f' \right)|\nabla H|^2 \, .
	\] Now \eqref{Evol-u-tion3} follows by substituting the above and \eqref{EvK} into \[
	\d_t\gamma = \frac{\d_t K}{H^n} - \frac{K}{H^n}\d_t H^n \, . \qedhere
	\]
	
\end{proof}

\begin{proposition}\label{PinchingProposition}
	There exists a constant $0 < C(n,\alpha) < 1/n^n$ such that if the initial hypersurface $\M_0$ is pinched in the sense that
	\begin{equation}\label{PinchingEquation}
	\frac{K(p)}{H(p)^n} \geq C(n,\alpha) \quad \forall p \in \M_0 \, ,
	\end{equation}
	then the pinching \eqref{PinchingEquation} is preserved by the flow.
\end{proposition}

Before proving Proposition \ref{PinchingProposition}, we need the following lemma:
\begin{lemma}\label{ExistenceOfEpsilon}
	There exists a constant $0 < C < 1/n^n$ such that if a convex hypersurface $\mathcal{N}$ satisfies \eqref{PinchingEquation} with this constant $C$, then there exists $0 < \varepsilon = \varepsilon(C) \leq 1/n$ such that
	\begin{align}
	\label{FirstEpsilon} h_{ij} \geq \varepsilon Hg_{ij} \\
	\intertext{and}
	\label{SecondEpsilon} \left| b^{ij} - \frac{n}{H}g^{ij} \right| \leq \frac{\varepsilon^2}{8\alpha H}
	\end{align} on $\mathcal{N}$. Moreover, $\varepsilon \longrightarrow 1/n$ as $C \longrightarrow 1/n^n$.
\end{lemma}

\begin{proof}
	Suppose $\mathcal{N}$ satisfies \eqref{PinchingEquation}, and let $\lambda_1,...,\lambda_n$ be the principal curvatures of $\mathcal{N}$ with the ordering $0 < \lambda_1 \leq ... \leq \lambda_n$. Define the quantities \[
	x_k = \frac{\lambda_k}{\sum_{l=1}^{n} \lambda_l} \, .
	\] Note that we also have $0 < x_1 \leq ... \leq x_n \leq 1$. To prove \eqref{FirstEpsilon}, we must first show that $x_1$ is bounded below. This is the case, since $\sum \lambda_l \geq \lambda_k$ for all $k$, and so \[
	0 < C \leq \frac{K}{H^n} = \prod_{l=1}^n x_l = \frac{\lambda_1...\lambda_n}{(\sum_{l=1}^n \lambda_l)^n} \leq \frac{\lambda_1...\lambda_n}{(\sum_{l=1}^n \lambda_l)\lambda_2...\lambda_n} = x_1 \, .
	\] Now take $\varepsilon$ to be the maximal constant such that $x_1 \geq \varepsilon$, which is equivalent to \eqref{FirstEpsilon}. The fact $\varepsilon \longrightarrow 1/n$ as $C \longrightarrow 1/n^n$ comes from the arithmetic-geometric mean inequality: as $C \longrightarrow 1/n^n$, all the principal curvatures tend to being equal, so $x_1 \longrightarrow 1/n$, and therefore the maximal $\varepsilon$ must itself tend to $1/n$. \\
	Now consider \eqref{SecondEpsilon}. By writing the inequality in terms of the principal curvatures, we find that it is equivalent to show \[
	\left| \frac{1}{n} - x_k \right| \leq \frac{\varepsilon^2}{8\alpha n}x_k \quad \forall k \, .
	\] To prove this, set $\varepsilon = 1/n - \delta$. We deal with the cases $x_k < 1/n$ and $x_k > 1/n$ separately (if $x_k = 1/n$ then the statement holds immediately). First assume that $x_k < 1/n$. By choosing $C$ close enough to $1/n^n$, we can make $\delta$ small enough so that $\delta + \sqrt[3]{8\alpha n\delta} \leq 1/n$. We therefore have \[
	\frac{1}{n} - \delta \leq x_k \implies \left| \frac{1}{n} - x_k \right| \leq \delta \leq \frac{1}{8\alpha n}\left( \frac{1}{n} - \delta \right)^3 = \frac{\varepsilon^2}{8\alpha n}\left( \frac{1}{n} - \delta \right) \leq \frac{\varepsilon^2}{8\alpha n}x_k \, ,
	\] as required. Secondly assume that $x_k > 1/n$. We first need to bound $x_k$ above by something depending on $\delta$. Since $\varepsilon$ is a lower bound for all the $x_l$, \[
	\varepsilon^{n-1} x_k \leq \prod_{l=1}^{n} x_l = \frac{K}{H^n} \leq \frac{1}{n^n} \implies x_k \leq \frac{1}{n^n\varepsilon^{n-1}} = \frac{1}{n(1-n\delta)^{n-1}} \, .
	\] By the binomial theorem it is possible to choose $C$ large enough (or equivalently, $\delta$ small enough) that \[
	x_k \leq \frac{1}{n(1-n\delta)^{n-1}} \leq \frac{1}{n} + n\delta \, .
	\] Thus we have \[
	\left| \frac{1}{n} - x_k \right| \leq n\delta \leq \frac{1}{8\alpha n} \left( \frac{1}{n} - \delta \right)^3 \leq \frac{\varepsilon^2}{8\alpha n} x_k \, ,
	\] provided $\delta + \sqrt[3]{8\alpha n^2\delta} \leq 1/n$. Thus, we have shown that for $C$ close enough to $1/n^n$, \eqref{SecondEpsilon} holds.
\end{proof}

We are now ready to prove Proposition \ref{PinchingProposition}.

\begin{proof}[Proof of Proposition \ref{PinchingProposition}]
	We use the parabolic maximum principle on the evolution equation \eqref{Evol-u-tion3} in Lemma \ref{Evol-u-tion}. Since $Hf' - f \geq 0$ by \eqref{Hf'-f}, and the identity $H^2 \leq n|A|^2$ holds for any hypersurface by the Cauchy-Schwartz inequality, we automatically have \[
	\frac{Hf' - f}{H}(n|A|^2 - H^2) \geq 0 \, .
	\] We now show that for the right choice of initial surface, \[
	\frac{Y^2}{H^2} + \frac{f''}{f'}\left( b^{ij} - \frac{n}{H}g^{ij} \right)\nabla_i H\nabla_j H \geq 0 \, .
	\] As in Corollary 4.3 of \cite{Chow} and Lemma 2.2 of \cite{Schulze2} we take the number $C(n,\alpha)$ to be the minimal constant such that $0 \leq C(n,\alpha) < 1/n^n$ and such that there exists $\varepsilon>0$ satisfying \[
	h_{ij} \geq \varepsilon Hg_{ij} \text{ and } \left| b^{ij} - \frac{n}{H}g^{ij} \right| \leq \frac{\varepsilon^2}{8\alpha H} \, .
	\] We can do this by Lemma \ref{ExistenceOfEpsilon}. Since on convex surfaces $b^{ij} \geq g^{ij}/H$ always holds, due to $|A|^2 \leq H^2$ and the fact that $h_{ij}$ is positive definite, we can also estimate \[
	Y^2 \geq \frac{|H\nabla_i h_{jl} - h_{jl}\nabla_i H|^2}{H^2} \geq \frac{1}{2}\varepsilon^2|\nabla H|^2
	\] by Lemma 2.3(ii) in \cite{Huisken1984}. Hence we have \[
	\frac{Y^2}{H^2} + \frac{f''}{f'}\left[ \left( b^{ij} - \frac{n}{H}g^{ij} \right)\nabla_i H\nabla_j H \right] \geq \left[ \frac{\varepsilon^2}{2} - 2\alpha\frac{\varepsilon^2}{8\alpha} \right]\frac{|\nabla H|^2}{H^2} \geq \frac{\varepsilon^2}{4}\frac{|\nabla H|^2}{H^2} > 0 \, .
	\] The desired result therefore follows from the maximum principle because the coefficient of $\gamma$ in \eqref{Evol-u-tion3} is positive.
\end{proof}

Following \cite{Schulze2}, we now want to show that at points where the mean curvature is large, the ratios between the principal curvatures approaches 1. To do this,  define
\begin{equation} \label{Definingg}
g := \frac{1}{n^n} - \frac{K}{H^n} \, \text{ and } \,  g_\sigma = g\phi(H) \, ,
\end{equation}
where $\phi(H)$ is a function to be chosen later with dependence on a small parameter $\sigma > 0$. Note that $0 \leq g \leq 1/n^n$ and $g(p,t) = 0$ if and only if the principal curvatures at $(p,t)$ are all equal.

\begin{lemma}\label{Evolutiong}
	The function $g_\sigma$ has satisfies the evolution equation
	\begin{align*}
	\frac{\d g_\sigma}{\d t} \, = {} & f'\Delta g_\sigma + \frac{f'}{\phi}\frac{H^n}{K}|\nabla g_\sigma|^2 + 2f' \left[ 1 - \frac{H\phi'}{\phi}\left( 1 + \frac{H^n}{K}g \right) \right]\inprod{\nabla g_\sigma}{\frac{\nabla H}{H}} \\
	& + f'\frac{H\phi'}{\phi}\left[ \frac{Hf''}{f'} - \frac{H\phi''}{\phi'} - 2\left( 1 - \frac{H\phi'}{\phi} \right) + \frac{H^n}{K}\frac{H\phi'}{\phi} \right]\frac{|\nabla H|^2}{H^2}g_\sigma \\
	& - \phi\gamma \left[ f'\frac{Y^2}{H^2} + f''\left( b^{ij} - \frac{n}{H}g^{ij} \right)\nabla_i H \nabla_j H \right] - \phi\gamma \frac{Hf' - f}{H}(n|A|^2 - H^2) + fg\phi'|A|^2 \, .
	\end{align*}
\end{lemma}

\begin{proof}
	We compute
	\begin{subequations}
	\begin{align}
	\nabla g = & {} -\nabla\gamma \, , \\
	\nabla g_\sigma = & {} \phi\nabla g + g\phi'\nabla H \, , \label{nablagsigma} \\
	\Delta g_\sigma = & {} -\phi\Delta\gamma + g(\phi''|\nabla H|^2 + \phi'\Delta H) + 2\phi'\inprod{\nabla g}{\nabla H} \, , \label{Laplacegsigma} \\
	\inprod{\nabla\gamma}{\nabla K} = & {} H^n|\nabla\gamma|^2 + \frac{nK\inprod{\nabla\gamma}{\nabla H}}{H} = H^n|\nabla g|^2 - \frac{nK\inprod{\nabla g}{\nabla H}}{H} \, , \label{gradugradK} \\
	|\phi\nabla g|^2 = & {} |\nabla g_\sigma|^2 - 2g\phi'\inprod{\nabla g_\sigma}{\nabla H} + g^2 (\phi')^2|\nabla H|^2 \, . \label{phigradgsquared}
	\end{align}
	\end{subequations}
	Using these,
	\begin{align*}
	\frac{\d g_\sigma}{\d t} = {} & \phi\frac{\d g}{\d t} + g\phi'\frac{\d H}{\d t} = -\phi\frac{\d\gamma}{\d t} + g\phi'\frac{\d H}{\d t} \\
	\stackrel{\eqref{Evol-u-tion3}}{=} {} & - f'\phi\Delta\gamma - f'\phi(n+1) \inprod{\nabla\gamma}{\frac{\nabla H}{H}} + f'\phi\frac{n-1}{nK}\inprod{\nabla\gamma}{\nabla K} + f'\phi\frac{H^n}{nK}|\nabla\gamma|^2 \\
	& - \phi\gamma\frac{Hf' - f}{H}(n|A|^2 - H^2) - f'\phi\gamma\frac{Y^2}{H^2} - f''\phi\gamma\left[ \left( b^{ij} - \frac{n}{H}g^{ij} \right)\nabla_i H\nabla_j H \right] \\
	& + g\phi' \left[ f'\Delta H  + f'' |\nabla H|^2 + f|A|^2 \right] \\
	\stackrel[\eqref{gradugradK}]{\eqref{Laplacegsigma}}{=} {} & f' \Delta g_\sigma - f'g\phi''|\nabla H|^2 - 2f'\phi' \inprod{\nabla g}{\nabla H} + 2f'\inprod{\phi\nabla g}{\frac{\nabla H}{H}} + \frac{f'}{\phi}\frac{H^n}{K}|\phi\nabla g|^2 \\
	& - \phi\gamma\frac{Hf' - f}{H}(n|A|^2 - H^2) - f'\phi\gamma\frac{Y^2}{H^2} - f''\phi\gamma\left[ \left( b^{ij} - \frac{n}{H}g^{ij} \right)\nabla_i H\nabla_j H \right] \\
	& + g\phi'\left[ f'' |\nabla H|^2 + f|A|^2 \right] \\
	\stackrel{\eqref{phigradgsquared}}{=} {} & f' \Delta g_\sigma + \frac{f'}{\phi}\frac{H^n}{K}|\nabla g_\sigma|^2 + 2f'\left( 1 - \frac{H\phi'}{\phi} \right)\inprod{\phi\nabla g}{\frac{\nabla H}{H}} - 2f'g\frac{H^n}{K}\frac{H\phi'}{\phi}\inprod{\nabla g_\sigma}{\frac{\nabla H}{H}} \\
	& - f'g\phi''|\nabla H|^2 + f''g\phi'|\nabla H|^2 + \frac{f'g^2(\phi')^2}{\phi}\frac{H^n}{K}|\nabla H|^2 \\
	& - \phi\gamma\frac{Hf' - f}{H}(n|A|^2 - H^2) + fg\phi'|A|^2 - f'\phi\gamma\frac{Y^2}{H^2} - f''\phi\gamma\left[ \left( b^{ij} - \frac{n}{H}g^{ij} \right)\nabla_i H\nabla_j H \right] \, .
	\end{align*}
	The evolution equation now follows by applying \eqref{nablagsigma} to the $\inprod{\phi\nabla g}{\nabla H/H}$ term.
\end{proof}

In order to prove the main theorem about $g_\sigma$, we also need the following result, proved as Lemma 2.5 in \cite{Schulze2}, which holds for any convex surface with $h_{ij} \geq \varepsilon Hg_{ij}$:

\begin{lemma}\label{LastTwoTerms}
	If a convex surface is such that $\lambda_i \geq \varepsilon H > 0$ for some $\varepsilon > 0$ and all $i = 1,...,n$, then there exists $\delta > 0$ such that \[
	\frac{n|A|^2 - H^2}{H^2} \geq \delta \left( \frac{1}{n^n} - \frac{K}{H^n} \right) \, .
	\]
\end{lemma}

We now make the choice \[
\phi(H) = (\ln\hat{H})^\sigma \, .
\] From this, and by applying the maximum principle, we can show the following theorem.

\begin{theorem}\label{PinchingResult}
	Let the initial hypersurface $\M_0$ satisfy the pinching condition \eqref{PinchingEquation}, and $g_\sigma$ be as in \eqref{Definingg}. Then there exists $\sigma > 0$ such that
	\begin{equation}\label{ScalingInvariantEquation}
	g_\sigma(p,t) \leq \max_{p\in \M_0}g_\sigma(p,0) \quad \forall(p,t) \in \M\times[0,T) \, .
	\end{equation}
\end{theorem}

\begin{proof}
	We deal with each term in the evolution equation for $g_\sigma$ which was computed in \ref{Evolutiong} separately. Define
	\begin{align*}
	L_1 \, = {} & \frac{f'}{\phi}\frac{H^n}{K}|\nabla g_\sigma|^2 + 2f' \left[ 1 - \frac{H\phi'}{\phi}\left( 1 + \frac{H^n}{K}g \right) \right]\inprod{\nabla g_\sigma}{\frac{\nabla H}{H}} \\
	L_2 \, = {} & f'\frac{H\phi'}{\phi}\left[ \frac{Hf''}{f'} - \frac{H\phi''}{\phi'} - 2\left( 1 - \frac{H\phi'}{\phi} \right) + \frac{H^n}{K}\frac{H\phi'}{\phi} \right]\frac{|\nabla H|^2}{H^2}g_\sigma \\
	L_3 \, = {} & - \phi\gamma \left[ f'\frac{Y^2}{H^2} + f''\left( b^{ij} - \frac{n}{H}g^{ij} \right)\nabla_i H \nabla_j H \right] \\
	L_4 \, = {} & - \phi\gamma \frac{Hf' - f}{H}(n|A|^2 - H^2) + fg\phi'|A|^2 \, .
	\end{align*}
	To apply the maximum principle we must have $L_2 + L_3 + L_4 \leq 0$. The term $L_1$ contains only gradient terms and can be ignored since we are interested in the maximum of $g_\sigma$. The fact that $|A|^2 \leq H^2$ and Lemmas \ref{ExistenceOfEpsilon} and \ref{LastTwoTerms} together imply that
	\begin{align*}
	L_3 \leq {} & - f'\phi\frac{C\varepsilon^2}{4}\frac{|\nabla H|^2}{H^2} \, , \\
	\intertext{and}
	L_4 \leq {} & - (C\alpha\delta - \sigma) g\frac{H^3}{\hat{H}}(\ln\hat{H})^{\sigma+\alpha-1} \, .
	\end{align*}
	Moreover, since $H^n/K \leq 1/C$ and $g \leq 1/n^n$, \[
	L_2 \leq f'\phi\frac{\sigma}{n^n}\left( 2\alpha + \sigma\left[ 1 + \frac{1}{Cn^n} \right] \right)\frac{|\nabla H|^2}{H^2} \, .
	\] Therefore we have 
	\begin{align*}
	\frac{\d g_\sigma}{\d t} - f'\Delta g_\sigma & \leq L_1 + f'\phi\frac{\sigma}{n^n}\left( 2\alpha + \sigma\left[ 1 + \frac{1}{Cn^n} \right] \right)\frac{|\nabla H|^2}{H^2} - f'\phi\frac{C\varepsilon^2}{4}\frac{|\nabla H|^2}{H^2} - (C\alpha\delta - \sigma) g\frac{H^3}{\hat{H}}(\ln\hat{H})^{\sigma+\alpha-1} \\
	& = L_1 + \frac{f'\phi}{n^n}\left( \sigma\left( 2\alpha + \sigma\left[ 1 + \frac{1}{Cn^n} \right] \right) - \frac{C\varepsilon^2 n^n}{4} \right) \frac{|\nabla H|^2}{H^2} - (C\alpha\delta - \sigma) g\frac{H^3}{\hat{H}}(\ln\hat{H})^{\sigma+\alpha-1} \, .
	\end{align*}
	Hence, by choosing \[
	0 < \sigma \leq \min\left\{ C\alpha\delta \, , \, \frac{-2\alpha + \sqrt{4\alpha^2 + C\varepsilon^2n^n + \varepsilon^2}}{2(1 + 1/Cn^n)} \right\} \, ,
	\] we see that the last three terms in the evolution equation for $g_\sigma$ in Lemma \ref{Evolutiong} are negative, which is the correct sign to apply the maximum principle, giving the result.
\end{proof}

\section{Convergence to a Sphere}
We now show that under a suitable scaling procedure near a singularity the limiting hypersurface becomes spherical. Assume that a singularity occurs at time $T$. Using the evolution equation \eqref{TimeChange4} from Proposition \ref{TimeChange} and $|A|^2 \leq H^2$, we obtain
\[
\frac{\d}{\d t}H_{\max} \leq H_{\max}^3 (\ln\hat{H}_{\max})^\alpha \, .
\] Define the variable $v(t)$ by \[
v(t) := H_{\max}(t) \quad \text{ so that } \quad \frac{\d v}{\d t} \leq v^3 (\ln\hat{v})^\alpha \, .
\] Let the function $J_\alpha : (0,\infty) \longrightarrow (-\infty,0)$ be a solution of the differential equation \[
J_\alpha'(x) = \frac{1}{x^3 (\ln\hat{x})^\alpha} \, ; \quad J_\alpha(x) \longrightarrow 0 \text{ as } x \longrightarrow \infty \, .
\] This function $J_\alpha$ has non-zero derivative on $(0,\infty)$, so is a bijection, and for large enough $x\in(0,\infty)$,
\begin{equation}\label{JSandwich}
-\frac{1}{2x^2} < J_\alpha(x) < -\frac{1}{3x^3} \, .
\end{equation}
For $0 \leq t < s \leq T$ we therefore have \[
T - t \geq s - t \geq \int_{v(t)}^{v(s)} \frac{1}{v^3(\ln\hat{v})^\alpha} \equiv J_\alpha(v(s)) - J_\alpha(v(t)) \, ,
\]
and since we know that $v(s) = H_{\max}(s) \longrightarrow \infty$ as $s \longrightarrow T$, the defining boundary condition for $J_\alpha$ gives us that $J_\alpha(v(s)) \longrightarrow 0$ as $s \longrightarrow T$. Thus \[
\frac{1}{T-t} \leq -\frac{1}{J_\alpha(v(t))} \, .
\] Define the function $G_\alpha : (0,\infty) \longrightarrow (0,\infty)$ by \[
G_\alpha(x) := -\frac{1}{J_\alpha(x)} \, .
\] For large enough $x\in(0,\infty)$, inequality \eqref{JSandwich} implies \[
2x^2 < G_\alpha(x) < 3x^3 \, .
\] $G_\alpha$ is also a bijection and is therefore invertible. If we denote its inverse by $G_\alpha^{-1}$, then we have $G_\alpha(x) \longrightarrow \infty$ and $G_\alpha^{-1}(x) \longrightarrow \infty$ as $x \longrightarrow \infty$. Finally, we have $G_\alpha^{-1}(x) \leq \sqrt{x/2}$. \\

We refer to a singularity as type-1 if there exists $C_0 > 0$ such that \[
H_{\max}(t) \leq C_0 G_\alpha^{-1}\left( \frac{1}{T-t} \right) \ \text{ for all } \ t \in [0,T) \, .
\] If no such $C_0$ exists we call it type-2. We consider a different scaling procedure for each type.

Assuming first that we have a type 1 singularity, we can choose a sequence of points and times $(x_k,t_k) \in \M^n \times [0,T-1/k]$ such that
\begin{equation}\label{ChoosingSequence1}
H(x_k,t_k) = \max_{\stackrel{t\leq T-1/k}{x\in\M^n}} H(x,t) \, .
\end{equation}
Now we rescale the surfaces by \[
\tilde{F}_k(\cdot,\tau) = \frac{F\left( \cdot,t_k + \tau(T-t_k-1/k) \right) - F(x_k,t_k)}{\varepsilon_k}
\] where \[
\tau \in \left[\, -\frac{t_k}{T-t_k-1/k} \, , \, 1 \,\right]
\ \ \text{ and } \ \ \varepsilon_k = \frac{1}{G_\alpha^{-1}\left( \frac{1}{T-t_k-1/k} \right)} \, .
\] With this choice of scaling, we have \[
\tilde{H}_k(\cdot,\tau) = \frac{H(\cdot,t_k+\tau(T-t_k-1/k))}{G_\alpha^{-1}\left( \frac{1}{T-t_k-1/k} \right)} \leq C_0\frac{G_\alpha^{-1}\left( \frac{1}{T - t_k -\tau(T-t_k-1/k)} \right)}{G_\alpha^{-1}\left( \frac{1}{T - t_k - 1/k} \right)} \, ,
\] which is bounded for large enough $k$ by the continuity of $G_\alpha^{-1}$.

Now assume we have a type-2 singularity, meaning \[
\frac{H_{\max}(t)}{G_\alpha^{-1}\left( \frac{1}{T-t} \right)} \longrightarrow \infty \ \text{ as } t \longrightarrow T \, .
\]

Choose a sequence of points and times $(x_k,t_k) \in \M^n \times [0,T-1/k]$ such that
\begin{equation}\label{ChoosingSequence2}
\frac{H(x_k,t_k)}{G_\alpha^{-1}\left(\frac{1}{T-\frac{1}{k}-t_k}\right)} = \max_{\stackrel{t\leq T-1/k}{x\in\M^n}} \frac{H(x,t)}{G_\alpha^{-1}\left(\frac{1}{T-\frac{1}{k}-t}\right)} \, .
\end{equation}

Now scale the surfaces according to \[
\tilde{F}_k(\cdot,\tau) = \frac{F\left( \cdot,t_k + \frac{\tau}{G_\alpha(1/\varepsilon_k)} \right) - F(x_k,t_k)}{\varepsilon_k}
\] where \[
\tau \in [\, -G_\alpha(1/\varepsilon_k)t_k \, , \, G_\alpha(1/\varepsilon_k)(T-t_k-1/k) \,]
\ \ \text{ and } \ \ \varepsilon_k = \frac{1}{H(x_k,t_k)} \, .
\]
Since we are considering a type-2 singularity, \eqref{ChoosingSequence2} implies that $1/\varepsilon_k \longrightarrow \infty$ as $k \longrightarrow \infty$.

We can now consider the mean curvature of the rescaled surfaces, $\tilde{H}_k(\cdot,\tau)$. This satisfies \[
0 \leq \tilde{H}_k(\cdot,\tau) = \varepsilon_k H(\cdot,t) \leq \frac{G_\alpha^{-1}\left( \frac{1}{T - t_k -1/k - \tau/G_\alpha\left( 1/\varepsilon_k \right)} \right)}{G_\alpha^{-1}\left( \frac{1}{T - t_k - 1/k} \right)} \, .
\] As before, the right hand side of this is bounded for large $k$ by the continuity of $G_\alpha^{-1}$. The boundedness of $\tilde{H}$ also gives bounds on the higher derivatives of curvatures. Thus, by the Arzela-Ascoli theorem, we can extract a subsequence of $\tilde{F}_k(\cdot,t)$ which converges uniformly on compact subsets of $\R^{n+1}\times\R$ to a limit $\tilde{F}_\infty(\cdot,\tau)$.

In both cases (type-1 and type-2), the estimate in Theorem \ref{PinchingResult} implies that the scaled limiting surface is spherical. This is because the quantity $K/H^n$ is scaling-invariant, and so \eqref{ScalingInvariantEquation} implies \[
0 \leq\frac{1}{n^n} - \frac{\tilde{K}}{\tilde{H}^n} = \frac{1}{n^n} - \frac{K}{H^n} = \frac{g_\sigma}{(\ln\hat{H})^\sigma} \leq \frac{\max_{p\in \M_0}g_\sigma(p,0)}{(\ln\hat{H})^\sigma} \longrightarrow 0 \text{ as } t \longrightarrow T \, ,
\] for all $p\in \M$. Therefore at all points on $\tilde{F}_\infty$ we have \[
\frac{1}{n^n} = \frac{\tilde{K}}{\tilde{H}^n} \, ,
\] which implies that $\tilde{F}_\infty$ is spherical.

\subsection*{Acknowledgements}
The author wishes to thank their supervisor Aram Karakhanyan for suggesting this problem and for his helpful comments during the preparation of this paper.

\end{document}